\documentclass[oneside,12pt]{amsart}
\usepackage[T2A]{fontenc}
\usepackage[cp1251]{inputenc}
\usepackage{amssymb}
\usepackage{amsxtra}
\usepackage{amsmath}
\usepackage{amstext,amsfonts,amscd}

\usepackage{amsthm}
\usepackage[all]{xy}

\usepackage{hyperref}
\usepackage{tikz,graphicx}
\usetikzlibrary{matrix,arrows}

\usepackage[russian,english]{babel}

\DeclareMathOperator{\GL}{GL}

\DeclareMathOperator{\Pic}{Pic}
\DeclareMathOperator{\Div}{Div}

\DeclareMathOperator{\Spin}{Spin}

\DeclareMathOperator{\RR}R

\DeclareMathOperator{\Br}{Br}

\DeclareMathOperator{\Spec}{Spec}

\DeclareMathOperator{\Gm}{{\mathbb G}_m}
\newcommand{\id}{\text{\rm id}}

\DeclareMathOperator{\coker}{coker\,}

\DeclareMathOperator{\divi}{div}

\DeclareMathOperator{\Char}{char}
\DeclareMathOperator{\Supp}{supp}
\newcommand{\Aff}{\mathbb {A}}
\newcommand{\Pro}{\mathbb {P}}

\newcommand{\CF}{F}

\DeclareMathOperator{\A1}{\Aff^1}

\newtheorem{lem}{Lemma}[section]
\newtheorem*{lem*}{Lemma}
\newtheorem*{thm*}{Theorem}
\newtheorem{thm}[lem]{Theorem}
\newtheorem{cor}[lem]{Corollary}
\newtheorem{defn}[lem]{Definition}
\theoremstyle{definition}{
\newtheorem{rem}[lem]{Remark}
}
\theoremstyle{definition}{

}

\let\l\left
\let\r\right

\newcommand{\Sm}{\mathbf{Sm}}
\newcommand{\Ab}{\mathbf{Ab}}
\DeclareMathOperator{\Groups}{\mathbf{Grp}}

\newcommand \xra {\xrightarrow }

\makeatletter

\@addtoreset{equation}{section}

\makeatother

\begin{document}

\selectlanguage{english}

\title{Transfers for non-stable $K_1$-functors of classical type}

\maketitle

\begin{center}
\vspace{20pt}
{\large A. Stavrova\footnote{
The author acknowledges support of the postdoctoral
grant 6.50.22.2014 ``Structure theory, representation theory and geometry of algebraic groups''
of St. Petersburg State University, and the RFBR grants 14-01-31515-mol\_a, 13-01-00429-a.}}\\
\vspace{10pt}
Department of Mathematics and Mechanics\\St. Petersburg State University\\
St. Petersburg, Russia\\
anastasia.stavrova@gmail.com
\end{center}

\begin{abstract}
Let $k$ be a field. Let $G$ be an absolutely almost simple simply connected $k$-group of type $A_l$, $l\ge 2$, or $D_l$, $l\ge 4$,
containing a 2-dimensional split torus. If $G$ is of type $D_l$, assume moreover that $\Char k\neq 2$.
We show that the Nisnevich sheafification of the non-stable $K_1$-functor $K_1^G$, also called the Whitehead
group of $G$, on the category of smooth $k$-schemes is $A^1$-invariant, and
has oriented weak transfers for affine varieties in the sense of Panin-Yagunov-Ross. If
$k$ has characteristic 0, this implies that the Nisnevich sheafification of $K_1^G$ is birationally invariant.

We also prove a rigidity theorem for $\A1$-invariant torsion presheaves with oriented weak transfers over
infinite fields. As a corollary, we conclude that $K_1^G(R)=K_1^G(k)$
whenever $R$ is a Henselian regular local ring with a coefficient field $k$.

\end{abstract}

\section{Introduction}


Let $\Sm_k$ be the category of
Noetherian smooth schemes of finite type over a field $k$. We will consider $\Sm_k$ as a site with Nisnevich topology,
and for any presheaf $F$ on $\Sm_k$, we
denote by $F_{Nis}$ its Nisnevich sheafification. Recall that a presheaf $F$ on $\Sm_k$ is called
\emph{$\Aff^1$-invariant},
if for any object $X\in\Sm_k$ the projection $X\times_k\Aff^1_k\to X$ induces an isomorphism $F(X\times_k\Aff^1_k)\cong F(X)$.

Let $G$ be a reductive group over $k$. We say that $G$
is isotropic if $G$ contains a (proper) parabolic $k$-subgroup $P$, or, equivalently, a non-central subgroup
isomorphic to $\Gm_{,k}$. If this is the case, for any $k$-scheme $X$
we set
$$
E_P(X)=\l<U_P(X),U_{P^-}(X)\r>\subseteq G(X),
$$
where $U_P$ and $U_{P^-}$ are the unipotent radicals of $P$ and any opposite parabolic subgroup $P^-$.
The quotient
$$
G(X)/E_P(X)=K_1^{G,P}(X)=W_P(X,G)
$$
is called the non-stable $K_1$-functor associated to $G$, or the Whitehead group of $G$.
Both names go back to Bass' founding paper~\cite{Bass}, where the case $G=\GL_n$ was considered.

It is known that the functor $K_1^{G,P}$ on affine $k$-schemes is independent of the choice
of a parabolic $k$-subgroup $P$ that intersects properly every semisimple normal subgroup of $G$.
If every semisimple normal subgroup of $G$ contains $(\Gm)^2$, then $K_1^{G,P}$
takes values in the category of groups. Also, it is $\A1$-invariant on $\Sm_k$ whenever
$k$ is perfect. We refer to~\cite{St-poly} for these and other basic properties of non-stable $K_1$-functors associated to
isotropic groups.

In the present paper we study the Nisnevich sheafification $K_{1,Nis}^{G,P}$ of $K_1^{G,P}$ on the
category $\Sm_k$, where $G$ a simply connected simple algebraic $k$-group of classical type $A_l$, $l\ge 2$, or $D_l$,
$l\ge 4$. Note that all our results also extend to groups of  type $B_l$ or $C_l$, $l\ge 2$,
however, for such groups $K_{1,Nis}^G$ is trivial on $\Sm_k$;
see~\cite{Gil} and Lemma~\ref{lem:inj} below.

The above-mentioned properties of $K_1^{G,P}$ imply that $K_{1,Nis}^{G,P}=K_{1,Nis}^G$
is group-valued, independent of the choice of $P$, and $\A1$-invariant.
We aim to prove that $K_{1,Nis}^G$ is a presheaf
with oriented weak transfers for affine varieties in the sense of the following definition of J.
Ross~\cite{Ross}, which is a slight
modification of a definition of I. Panin and S. Yagunov~\cite{PaY}.

\begin{defn}\cite[Definition 2.1]{Ross}\label{def:transfer}
Let $\mathcal{A}$ be an additive category, and let $F:\Sm_k^{op}\to\mathcal{A}$ be a presheaf. Assume that $F$ is additive
in the sense that $F(X_1\coprod X_2)\cong F(X_1)\oplus F(X_2)$. We say that $F$ has \emph{oriented weak transfers}
if for any $X,Y\in \Sm_k$, any finite flat generically \'etale morphism $f:X\to Y$,
 and any closed embedding of $Y$-schemes $\tau:X\hookrightarrow Y\times_k\Aff_k^n$ with trivial normal bundle,
there is a map $f_*^\tau:F(X)\to F(Y)$
satisfying the following properties.
\begin{enumerate}
\item The $f_*^\tau$'s are compatible with disjoint unions: if $X=X_1\coprod X_2$, $f_i:X_i\to Y$ and
$\tau_i:X_i\hookrightarrow Y\times_k\Aff_k^n$, $i=1,2$ are the
morphisms induced by $f$ and $\tau$, then the following diagram commutes:
\begin{equation*}
\xymatrix@R=15pt@C=40pt{
F(X)\ar[d]_{\cong}\ar[r]^{f_*^\tau}&F(Y)\\
F(X_1)\oplus F(X_2)\ar[ru]_{(f_{1*}^{\tau_1},f_{2*}^{\tau_2})}&\\
}
\end{equation*}

\item The $f_*^\tau$’s are compatible with sections $s:Y\to X$ which are isomorphisms
onto connected components of $X$. In the notation of the previous property,
if $s$ is an isomorphism onto $X_1$, then we require $f_{1*}^{\tau_1} = s^*$; in particular,
$f_{1*}^{\tau_1}$ is independent of $\tau_1$.

\item For any morphism $g:Y'\to Y$ which is either smooth, or a closed embedding of a principal smooth divisor,
the following diagram commutes:
\begin{equation*}
\xymatrix@R=15pt@C=40pt{
F(X)\ar[d]_{f_*^\tau}\ar[r]^{g'^*}&F(X\times_Y Y')\ar[d]^{(f')_*^{\tau'}}\\
F(Y)\ar[r]^{g^*}&F(Y')\\
}
\end{equation*}
Here $f':X\times_Y Y'\to Y'$,
$g':X\times_Y Y'\to X$, and $\tau':X'\hookrightarrow Y'\times_k\Aff_k^n$ are the natural morphisms obtained by base change.

\item $f_*^\tau$ is compatible with the addition of irrelevant summands: if
$$
(\tau,0):X\hookrightarrow Y\times_k\Aff_k^n\times_k\A1
$$
is the embedding induced by $\tau$ and $0\hookrightarrow\A1$, then $f_*^{\tau}$ and $f_*^{(\tau,0)}$
coincide.

\end{enumerate}
If $\mathcal{A}$ is the category of abelian groups $\Ab$, then all transfer maps are required to be group
homomorphisms.

One says that a presheaf $F:\Sm_k\to\mathcal{A}$ has \emph{weak transfers for affine varieties}, if
we are given weak transfers as described above whenever $X$ and $Y$ are affine.
\end{defn}


I. Panin and S. Yagunov introduced oriented weak transfers in order to further generalize
Suslin's rigidity theorem for $K$-theory with finite
coefficients~\cite{Sus83,GTh,Gab92}, which was
previously extended to $\Aff^1$-invariant torsion presheaves with transfers in the sense of Suslin--Voevodsky~\cite{SuVo,Vo-transf}.
It is proved in~\cite{PaY}
that orientable cohomology theories over algebraically closed fields possess oriented weak
transfers, and hence satisfy a rigidity theorem. These results were later generalized to presheaves with non-oriented weak transfers
over algebraically closed fields~\cite{Ya-II}, and to $\Aff^1$-stably representable cohomology theories over
infinite fields~\cite{HY}.

In~\cite{Ross} J. Ross, assuming that $k$ has characteristic $0$, extends to presheaves with weak transfers
the main results of~\cite{Vo-transf}. Namely, he shows that
Zariski and Nisnevich sheafifications of
an $\A1$-invariant presheaf on $\Sm_k$ with oriented weak transfers for affine varieties coincide,
and their Zariski and Nisnevich cohomology functors are again $\A1$-invariant presheaves with oriented
weak transfers. Also, such presheaves have Gersten resolutions.

Our main result is the following theorem.

\begin{thm}\label{thm:transf}
Let $k$ be a field. Let $G$ be
a simply connected simple algebraic group over $k$ of type $A_l$, $l\ge 2$, or $D_l$, $l\ge 3$,
 such that $G$ contains $(\Gm_{,k})^2$. If $G$ is of type $D_l$, assume moreover that $\Char k\neq 2$.
Then the functor $K_{1,Nis}^G(-)$ on the category $\Sm_k$
takes values in $\Ab$, and is an $\A1$-invariant functor with oriented weak transfers
for affine varieties.
\end{thm}

Note that, according to a recent result of A. Asok, M. Hoyois and M. Wendt~\cite[Corollaries 4.3.3 and 4.3.4]{AHW15},
for any infinite field $k$ and any isotropic reductive $k$-group $G$ the Zariski sheafification $K_{1,Zar}^G$ on $\Sm_k$
coincides with $K_{1,Nis}^G$. However, our proof of Theorem~\ref{thm:transf} uses Nisnevich topology, so
we keep the notation $K_{1,Nis}^G$.

Our construction of oriented weak transfers for $K_{1,Nis}^G$ is inspired by the construction of the norm homomorphism for $R$-equivalence
class groups of classical groups due to V. Chernousov and A. Merkurjev~\cite{ChM-su}. In particular,
their norm homomorphisms are precisely the oriented weak transfer maps for finite extensions of fields. By construction,
our transfers are independent of the choice of a closed emebedding $\tau$ of Definition~\ref{def:transfer}.
The characteristic assumption in the $D_l$ case is solely due to the fact that Merkurjev's norm principle
is not known for groups of type $D_l$ in characteristic $2$; see Theorem~\ref{thm:transf-gen} below.

We note that it should be possible to assume in Theorem~\ref{thm:transf}
that $G$ contains $\Gm_{,k}$ instead of $(\Gm_{,k})^2$.
In fact, one can show that for any isotropic simply connected simple algebraic group
$H$ of classical type over $k$ there is a group $G$ over $k$ satisfying the assumptions of Theorem~\ref{thm:transf}
and such that $K_{1,Nis}^G=K_{1,Nis}^{H}$ on $\Sm_k$. This follows from the
stabilization of non-stable $K_1$-functors of classical types on local rings~\cite{Yu-stab}. However,
in order to apply the latter result, we would have had to invoke the voluminous language of
Petrov's unitary groups which is unnecessary for our results here, so we chose not to do that.




It is not hard to prove that $K_{1,Nis}^G$ is not only $\Aff^1$- but also $\Gm$-invariant (see Lemma~\ref{lem:inj}).
The results of~\cite{Ross} and Theorem~\ref{thm:transf} then imply that $K_{1,Nis}^G$ is a birationally invariant sheaf.

\begin{cor}\label{cor:bir}
In the setting of Theorem~\ref{thm:transf}, assume moreover that $k$ has characteristic $0$. Then for
any smooth $k$-scheme $X$, one has
$$
K_{1,Nis}^G(X)\cong \prod_{i=1}^k K_{1,Nis}^G(k(X_i)),
$$
where $X_i$, $1\le i\le k$, are the connected components of $X$.
In particular, $H^n_{Nis}(X,K_{1,Nis}^G)=0$ for any n>0.
\end{cor}

Note that Corollary~\ref{cor:bir} implies that $H^n_{Nis}(X,K_{1,Nis}^G)=0$ for any $n>0$
and any smooth $k$-scheme $X$ (e.g. by a lemma of J. Riou~\cite[Lemma 4.1.4]{KL}); that is,
$K_{1,Nis}^G$ is strictly $\Aff^1$-invariant in the sense of F. Morel~\cite{Mo-book}.
The same result was also obtained by A. Asok, M. Hoyois and M. Wendt for the case of an arbitrary perfect field $k$
using different methods~\cite{AHW15}.

One can also extend to $\Aff^1$-invariant presheaves with oriented weak transfers the
Henselian local ring version of Suslin's rigidity theorem
by closely imitating the proofs of rigidity theorems in~\cite{SuVo,Gab92,HY}; see Theorem~\ref{thm:HY3.1} below.
As a corollary, we obtain  the following result. The corresponding statement was previously known
for general isotropic simply connected groups $G$ over Henselian
discrete valuation rings~\cite{Gil-spec},\cite[Corollaire 7.3]{Gil}, and for groups $G$ of type $A_l$
over arbitrary Henselian local rings~\cite{Haz-hp,Haz-hpu}.

\begin{cor}\label{cor:rigidity}
In the setting of Theorem~\ref{thm:transf}, assume that $G$ is of type $D_l$, $l\ge 4$, and  $\Char k\neq 2$.
Let $(R,m)$ be a Henselian regular local ring such that $k\cong R/m$ is a coefficient subfield of $R$.
Then there is a natural isomorphism
$$
K_{1}^G(R)\cong K_{1}^G(k).
$$
\end{cor}


I am indebted to Ivan Panin for pointing out that norm maps for algebraic tori can be defined
not only for field extensions, but also for finite flat morphisms of affine schemes. Also, I would like to thank
Alexey Ananyevskiy for several helpful comments on rigidity theorems.

\section{Nisnevich sheafification of a non-stable $K_1$-functor}

In order to prove Theorem~\ref{thm:transf}, we need to consider non-stable $K_1$-functors in a more general
situation than in the introduction. Let $R$ be a commutative ring with 1, and let $G$ be a reductive scheme over $R$ in the sense
of~\cite{SGA3}.

\begin{defn}
A parabolic subgroup $P$ in $G$ is called
\emph{strictly proper}, if it intersects properly every normal semisimple subgroup scheme of $G$.
\end{defn}

If $R$ is local, then $G$ contains a strictly proper parabolic $R$-subgroup if and only if
every semisimple normal $R$-subgroup scheme of $G$ contains $\Gm_{,R}$ by~\cite[Proposition 6.16]{SGA3}.

Assume that $G$ has a parabolic $R$-subgroup $P$ that is strictly proper. Since the base $\Spec R$ is affine,
$P$ has a Levi $R$-subgroup $L$, and there is a unique opposite parabolic $R$-subgroup $P^-$ of $G$
such that $P\cap P^-=L$~\cite[Exp. XXVI, Cor. 2.3 and Th. 4.3.2]{SGA3}. For any $R$-scheme $X$,
we set
$$
E_P(X)=\l<U_P(X),U_{P^-}(X)\r>\subseteq G(X),
$$
where $U_P$ and $U_{P^-}$ are the unipotent radicals of $P$ and $P^-$. Since any two Levi $R$-subgroups of $P$
are conjugate by an element of $U_P(R)$ by~\cite[Exp. XXVI, Cor. 1.8]{SGA3}, the group $E_P(X)$ is indeed
independent of the choice of $L$ and $P^-$.
The quotient
$$
G(X)/E_P(X)=K_1^{G,P}(X)=W_P(X,G)
$$
is again called the non-stable $K_1$-functor associated to $G$, or the Whitehead group of $G$. It is a functor
on the category of $R$-schemes.

If $R$ is a semilocal ring, or if every semisimple normal subgroup of $G$ contains
$(\Gm_{,k})^2$ locally in Zariski topology on $\Spec R$,
then for any $R$-algebra $A$, the group $E_P(A)=E(A)$ is independent of the choice of a strictly proper
parabolic $R$-subgroup $P$~\cite{SGA3,PS}, see~\cite[Theorem 2.1]{St-poly}.
In this case $K_1^{G,P}(A)$ is a group, since a
conjugate of a strictly proper parabolic subgroup is a strictly proper parabolic subgroup.

For any Noetherian commutative ring $R$ of finite dimension, we denote by $\Sm_R$ the category of
Noetherian smooth schemes of finite type over $R$. We consider $\Sm_R$ as a site with Nisnevich topology.
For any reductive group scheme $G$ over $R$ with a parabolic $R$-subgroup $P$ we denote by $K_{1,Nis}^{G.P}$
the Nisnevich sheafification of $K_1^{G,P}$ on $\Sm_R$.
Note that  $K_1^{G,P}$ extends by continuity to the category of Noetherian essentially smooth $R$-schemes.
This extension is well-defined thanks to~\cite[Corollaire 8.13.2]{EGAIV-3}, and compatible with Nisnevich
sheafification e.g. by~\cite[\href{http://stacks.math.columbia.edu/tag/00XI}{Tag 00XI}]{Stacks}.

\begin{lem}\label{lem:K1group}
Let $R$ be a smooth $k$-algebra, where $k$ is a field.
Let $G$ be a reductive group scheme over $R$ having a strictly
proper parabolic $R$-subgroup.
Then the functor $K_{1,Nis}^{G,P}=K_{1,Nis}^G$ on $\Sm_R$
is independent of the choice of a strictly proper parabolic $R$-subgroup $P$ of $G$, and takes values
in the category of groups.
\end{lem}
\begin{proof}
Let $P$ be a strictly proper parabolic $R$-subgroup of $G$. Since the presheaf $E_P$ is a subgroup presheaf
of $G$, and the sheafification functor (on presheaves of sets) is exact, its Nisnevich sheafification $E_{P,Nis}$
is a subgroup presheaf of the Nisnevich sheaf $G$.
If $Q$ is another strictly proper parabolic $R$-subgroup of $G$, then
$E_{P,Nis}=E_{Q,Nis}$ as sheaves on the category of smooth $R$-schemes, since they are both subsheaves of $G$
and coincide on stalks (which are local rings). In particular, $E_{P,Nis}$ is a subsheaf of normal subgroups of the sheaf $G$ on the category
of all smooth $R$-schemes.

On the other hand, $K_{1,Nis}^{G,P}$ is also the Nisnevich sheafification of the quotient
$G/E_{P,Nis}$. Therefore, $K_{1,Nis}^{G,P}=K_{1,Nis}^{G}$ is group-valued and independent of the choice of $P$.
\end{proof}

Let $k$ be a field, and let $R$ be a regular ring containing $k$. Let $G$ be
a reductive group over $k$, such that every semisimple normal subgroup of $G$ contains
$(\Gm_{,k})^2$. If $k$ is perfect, by~\cite[Theorem 1.3]{St-poly}  the natural inclusion $R\to R[t]$ induces
an isomorphism
\begin{equation}\label{eq:HI}
K_1^G(R)\cong K_1^G(R[t]).
\end{equation}

Assume that $k$ is infinite and $R$ is local regular $k$-algebra, and $K$ is the field of fractions of $R$.
Assume moreover that $k$ is perfect or that $A$ is a local ring of a smooth algebraic variety over $k$.
Then the natural map
$$
K_1^G(R)\to K_1^G(K)
$$
is injective by~\cite[Theorem 1.4]{St-poly}. We extend this result to the case of a finite $k$.

\begin{lem}\label{lem:fin-inj}
Let $k$ be a finite field, and let $G$ be
a reductive group over $k$, such that every semisimple normal subgroup of $G$ contains
$(\Gm_{,k})^2$. Let $R$ be a regular local ring containing $k$, and let $K$ be the field of fractions of $R$. Then
$$
K_1^G(R)\to K_1^G(K)
$$
is injective.
\end{lem}
\begin{proof}
The embedding $k\to R$ is geometrically regular, since $k$ is perfect. Hence by
by Popescu's theorem~\cite{Po90,Swan} $R$ is a filtered direct limit of regular local $k$-algebras essentially
of finite type. Thus, since $K_1^G$ commutes with filtered direct limits, we can assume from the start that
$R$ is a regular local $k$-algebra essentially of finite type. Then $R\otimes_k k((t))$ is
a regular algebra essentially of finite type over the infinite field $k((t))$. Let $m$ be the maximal ideal of $R$,
and let $n$ be the maximal
ideal of $R\otimes_k k((t))$ containing $m\otimes_k k((t))$. Set $R'=\bigl(R\otimes_k k((t))\bigr)_n$.
Let $K'$ denote the field of fractions of $R'$.
The map
\begin{equation}\label{eq:RK'}
K_1^G(R')\to K_1^G(K')
\end{equation}
is injective by~\cite[Theorem 1.4]{St-poly}.

Let $A$ be one of $R$, $K$, and let $A'$ be $R'$ or $K'$ respectively.
Note that
By~\cite[Theorem 1.3]{St-poly} one has $K_1^G(A[t])=K_1^G(A)$. Then
by~\cite[Corollary 3.4]{St-serr} the natural homomorphism
$$
K_1^G(A[t,t^{-1}])\to K_1^G\bigl(A((t))\bigr)
$$
is injective.
Since the natural inclusion $A\to A[t,t^{-1}]$ has a section, we conclude that
$$
K_1^G(A)\to K_1^G\bigl(A((t))\bigr)
$$
is injective. Since the natural homomorphism $A\to A((t))$ factors through $A'$, this implies that
$$
K_1^G(A)\to K_1^G(A')$$
is also injective. Then the injectivity of~\eqref{eq:RK'} implies that of $K_1^G(R)\to K_1^G(K)$.
\end{proof}



\begin{lem}\label{lem:inj}
Let $k$ be a field, and let $G$ be
a reductive group over $k$, such that every semisimple normal subgroup of $G$ contains
$(\Gm_{,k})^2$.

$(i)$ For any $X\in\Sm_k$, the natural map
$$
K_{1,Nis}^G(X)\to \prod_{i=1}^k K_{1,Nis}^G(k(X_i))
$$
is injective, where $X_i$, $1\le i\le n$, are the connected components of $X$.

$(ii)$ If $G$ is simply connected semisimple, then the functor $K_{1,Nis}^G$ on $\Sm_k$ is $\A1$-invariant
and  $\Gm$-invariant, that is, $K_{1,Nis}^G(\Gm_{,X})\cong K_{1,Nis}^G(X)\cong K_{1,Nis}^G(\Aff^1_X)$
for any $X\in\Sm_k$.

\end{lem}
\begin{proof}
$(i)$  Note that the claim folds for $X=\Spec R$, where $R$ is a local Henselian essentially smooth $k$-algebra.
Indeed, if $k$ is finite, it holds by Lemma~\ref{lem:fin-inj}. If $k$ is infinite, then by~\cite[Theorem 1.4]{St-poly}
the map $K_1^G(A)\to K_1^G(F)$ is injective for any ring $A$ with the fraction field $F$, whenever
$A$ is a local ring of a smooth $k$-variety.
Any local essentially smooth $k$-algebra $R$ is a direct limit of such local rings, and $K_1^G$ commutes with
direct limits, hence $K_1^G(R)\to K_1^G(K)$ is injective.

Consider the presheaf $F$ on $\Sm_k$ given by
$$
F(X)=\prod_{i=1}^k K_{1,Nis}^G(k(X_i)).
$$
The presheaf $F$ is a Nisnevich sheaf, since, clearly, it takes elementary Nisnevich squares to cartesian
squares. Then the natural morphism of functors $K_{1,Nis}^G\to F$ is injective, since it is injective on stalks.

$(ii)$ We first prove that $K_{1,Nis}^G(R)\cong K_{1,Nis}^G(R[t])$ for any
essentially smooth $k$-domain $R$.
First, assume that $R=K$ is a field extension
of $k$. By $(i)$, the homomorphism
$$
f:K_{1,Nis}^G(K[t])\to K_{1,Nis}^G(K(t))
$$
is injective. On the other hand, $f$ is surjective, since the natural homomorphism
$$
g:K_{1,Nis}^G(K)=K_1^G(K)\to K_1^G(K(t))=K_{1,Nis}^G(K(t))
$$
is an isomorphism by~\cite[Th\`eor\'eme 5.8 and Th\`eor\'eme 7.2]{Gil}. This implies that $f$ is an isomorphism.
Therefore, the map
$$
f^{-1}\circ g : K_{1,Nis}^G(K)\to K_{1,Nis}^G(K[t])
$$
is an isomorphism.

Now assume that $R$ is an essentially smooth $k$-domain, and let $K$ be its field of fractions. By $(i)$, the natural map
$$
K_{1,Nis}^G(R[t])\to K_{1,Nis}^G(K(t))
$$
is injective. Then the map $K_{1,Nis}^G(R[t])\to K_{1,Nis}^G(K[t])$ is also injective. Then the commutative
diagram
\begin{equation*}
\xymatrix@R=15pt@C=30pt{
K_{1,Nis}^G(R[t])\ar[r]^{t\mapsto 0}\ar[d]&K_{1,Nis}^G(R)\ar[d]\\
K_{1,Nis}^G(K[t])\ar[r]^{t\mapsto 0}& K_{1,Nis}^G(K)\\
}
\end{equation*}
implies that the map $K_{1,Nis}^G(R[t])\xrightarrow{t\mapsto 0}K_{1,Nis}^G(R)$ is injective, and hence an isomorphism.

To finish the proof of $(ii)$, consider an open cover $X=\bigcup_i U_i$, where $U_i$ are smooth connected affine $k$-schemes.
Then $\Aff^1_X=\bigcup_i \Aff^1_{U_i}$ is also an open cover. Assume that
$$
x\in\ker(K_{1,Nis}^G(\Aff^1_X)\xrightarrow{f_0} K_{1,Nis}^G(X)),
$$
where $f_0$ is the "restriction to 0"{} homomorphism. Then $x|_{U_i}=1$ for any $U_i$, since by the above
$$
K_{1,Nis}^G(\Aff^1_{U_i})\cong K_{1,Nis}^G(U_i).
$$
Since $K_{1,Nis}^G$ is a Zariski sheaf, this implies that $x=1$.

The proof of $\Gm$-invariance is the same as the proof of $\Aff^1$-invariance, except that we show that the "restriction to $1$"{}
homomorphism $K_{1,Nis}^G(\Gm_{,X})\to K_{1,Nis}^G(X)$ is injective.
\end{proof}

We will later need one more technical lemma.

\begin{lem}\label{lem:K1G'}
Let $k$ be a field, let $R$ and $S$ be two essentially smooth $k$-algebras, and let $f:R\to S$
be a finite flat $k$-algebra homomorphism.
Let $G$ be a reductive group over $k$,  having a strictly proper parabolic $k$-subgroup.
Set $G'=R_{S/R}(G_S)$. Then $G'$ is a reductive group scheme over $R$ having a strictly proper parabolic
$R$-subgroup, and for any essentially smooth $R$-algebra $A$ one has
$$
K_{1,Nis}^{G'}(A)=K_{1,Nis}^G(A\otimes_R S),
$$
where $K_{1,Nis}^{G'}$ and $K_{1,Nis}^G$ are the Nisnevich sheafifications of the corresponding non-stable
$K_1$-functors on $\Sm_R$ and $\Sm_k$ respectively.
\end{lem}
\begin{proof}
The group scheme $G'$ over $R$ is, clearly, affine; it is smooth by~\cite[\S 7.6, Proposition 5]{Neron-book}.
Its geometric fibers are reductive groups, hence it is a reductive group scheme. If $P$ is a strictly
proper parabolic subgroup of $G$, then $R_{S/R}(P_S)$ is a strictly
proper parabolic subgroup of $G'$ (this follows from the description of all semisimple normal
subgroups of $G$ in~\cite[Exp. XXIV, \S 5]{SGA3}). Clearly, for any essentially smooth $R$-algebra $A$
the algebra $A\otimes_R S$ is essentially smooth over $S$, and hence over $k$. By definition, we have
$K_1^{G'}(A)=K_1^G(A\otimes_R S)$. Hence there is a morphism of presheaves on the category
of smooth $R$-algebras $K_1^{G'}(-)\to K_{1,Nis}^G(-\otimes_R S)$. The restriction of $K_{1,Nis}^G(-\otimes_R S)$
to $\Sm_R$ is a Nisnevich sheaf e.g. by~\cite[\href{http://stacks.math.columbia.edu/tag/00XI}{Tag 00XI}]{Stacks}, hence this morphism factors through a morphism
$$
\phi:K_{1,Nis}^{G'}(-)\to K_{1,Nis}^G(-\otimes_R S).
$$
Assume that $A$ is a Henselian local ring. Since $A\otimes_R S$ is finite over $A$, we conclude that
$A\otimes_R S$ is a finite product of Henselian local rings. Then $K_{1,Nis}^G(A\otimes_R S)=K_1^G(A\otimes_R S)$.
Therefore, $\phi$ is an isomorphism on Nisnevich stalks, hence an isomorphism.
\end{proof}





\section{Proof of Theorem~\ref{thm:transf} and Corollary~\ref{cor:bir}.}

Our construction of transfers for non-stable $K_1$-functors is based on the construction of norm maps
for $R$-equivalence class groups due to  V. Chernousov and A. Merkurjev~\cite{ChM-su}.
We briefly recall their result.

\begin{rem}\label{rem:ChM}
Note that the groups that are routinly denoted by $G,H,G'$, and $H'$ in~\cite{ChM-su} we denote by $H',G',H$, and $G$
respectively, in order to secure compatibility with the statement of Theorem~\ref{thm:transf}.
\end{rem}

We will need the following notions.

\begin{defn}
Let $X$ be an algebraic variety over a field $k$.
Denote by $k[t]_{(t),(t-1)}$ the semilocal ring
of the affine line $\Aff^1_k$ over $k$ at the points $0$ and $1$. Two points $x_0,x_1\in X(k)$ are
called \emph{directly $R$-equivalent}, if there is $x(t)\in X\bigl(k[x]_{(x,x-1)}\bigr)$ such that $x(0)=x_0$ and $x(1)=x_1$.
The \emph{$R$-equivalence relation} on $X(k)$ is the equivalence relation generated by direct $R$-equivalence.
The \emph{$R$-equivalence class group $G(k)/RG(k)$} of an algebraic $k$-group $G$ is the quotient of $G(k)$
by the $R$-equivalence class $RG(k)$ of the neutral element $1_G\in G(k)$.
\end{defn}

It is easy to see that the $RG(k)$ is a normal subgroup
of $G(k)$, so $G(k)/RG(k)$ is indeed a group. If $G$ has a proper parabolic subgroup $P$
over $k$, then all elements of $E_P(k)$ are $R$-equivalent to $1$, so $K_1^{G,P}(k)$ surjects
onto $G(k)/RG(k)$.

\begin{defn}
Let $k$ be a field, and let $H\xrightarrow{\mu} T$ be a homomorphism of algebraic $k$-groups where $T$
is a commutative $k$-group. Let $F$ be a field extension of $k$. We say that \emph{$\mu$ satisfies Merkurjev's norm principle over $F$}, if for any \'etale
$F$-algebra $E$ the standard norm homomorphism $T(E)\xrightarrow{N_{E/F}} T(F)$ satisfies
$$
N_{E/F}\circ\mu\bigl(H(E)\bigr)\subseteq \mu\bigl(H(F)\bigr).
$$
\end{defn}

Let $H\xrightarrow{\mu} T$ be a homomorphism of algebraic $k$-groups where $T$
is a $k$-torus, and $H$ is a reductive $k$-group. Set $G=\ker\mu$. Assume that $F$ is an infinite field
extension of $k$ such that $\mu$ satisfies Merkurjev's norm principle over $F$ and $H(F)/RH(F)=1$.
Let $E/F$ be a finite separable field extension.
Chernousov and Merkurjev proved~\cite[\S 4]{ChM-su} that under these assumptions $N_{E/F}$
extends to ``a norm homomorphism''
$$
N_{E/F}: H(E)/RG(E)\to H(F)/RG(F),
$$
and the latter homomorphism restricts to a correctly defined homomorphism
\begin{equation}\label{eq:Gnorm}
N_{E/F}: G(E)/RG(E)\to G(F)/RG(F).
\end{equation}

Assume in addition that $G$ is a reductive $F$-group. Then $\mu$ is known to satisfy Merkurjev's norm
principle over $F$ whenever $F$ is perfect~\cite[Theorem 3.9]{M-norm}, or $H$ splits over a finite field extension of $F$ of
degree coprime to $\Char F$~\cite[Remark 4.1]{ChM-su}, as well as in many other cases, see e.g.~\cite{MeBa}.

If, moreover, $G$ semisimple and simply connected, and contains a strictly proper parabolic $k$-subgroup $P$,
then for any field extension $F$ of $k$ one has
$$
K_1^{G,P}(F)=K_1^G(F)\cong G(F)/RG(F)
$$
by~\cite[Th\'eor\`eme 7.2]{Gil}.
Thus, the map~\eqref{eq:Gnorm} becomes a group homomorphism
$$
N_{E/F}:K_1^G(E)\to K_1^G(F).
$$

We will deduce Theorem~\ref{thm:transf} from the following more general result.

\begin{thm}\label{thm:transf-gen}
Let $k$ be an infinite field. Let $G$ be a simply connected semisimple group over $k$ such
that every semisimple normal subgroup of $G$ contains $(\Gm_{,k})^2$. Assume that there exists a $k$-rational
reductive $k$-group $H$ and a $k$-torus $T$ such that $G$ fits into a short exact sequence of $k$-group
homorphisms
\begin{equation}\label{eq:seq}
1\to G\to H\xrightarrow{\mu} T\to 1.
\end{equation}
Assume also that $\mu$ satisfies Merkurjev's norm principle over every field extension $F$ of $k$.
Then the functor $K_{1,Nis}^G$ on $\Sm_k$ takes values in $\Ab$, and has transfer homomorphisms
$$
f_*:K_{1,Nis}^G(S)\to K_{1,Nis}^G(R),
$$
defined for any pair $R$, $S$ of essentially smooth $k$-algebras
and any finite flat generically \'etale $k$-algebra homomorphism $f:R\to S$, such that the following properties are satisfied.
\begin{enumerate}
\item Assume that $S=S_1\times S_2$
is a product of two regular $k$-algebras, and let $f_1:R\to S_1$ and $f_2:R\to S_2$
be the natural maps. Then $f_*=({f_1}_*,{f_2}_*)$. If, moreover,
$f_1:R\xrightarrow{\cong} S_1$ is a $k$-algebra isomorphism, then ${f_1}_*={((f_1)^{-1})}^*$.

\item For any $k$-algebra homomorphism $g:R\to R'$
the following diagram commutes:
\begin{equation*}
\xymatrix@R=15pt@C=40pt{
K_{1,Nis}^G(R)\ar[r]^{g^*}&K_{1,Nis}^G(R')\\
K_{1,Nis}^G(S)\ar[u]_{f_*}\ar[r]^{g'^*}&K_{1,Nis}^G(S\otimes_R R')\ar[u]_{(f')_*}\\
}
\end{equation*}
Here $f'$ and $g'$ are the natural homomorphisms obtained by base change.

\item If $f:E\to F$ is a finite separable extension of fields essentially smooth over $k$, then
$$
f_*=N_{E/F}:G(E)/RG(E)\to G(F)/RG(F)
$$
is the Chernousov--Merkurjev norm homomorphism~\cite[p. 187]{ChM-su}.
\end{enumerate}

\end{thm}

We will need the following technical lemma.
\begin{lem}\label{lem:rat-surj}
Let $A$ be a semilocal ring such that all residue fields of $A$ are infinite. Let $X$ be an affine group scheme
over $A$. Assume that $X$ is rational,
i.e. contains a dense open $A$-subscheme $V$ which is isomorphic to an open subscheme of $\Aff_A^n$ for some
$n\ge 1$. Then

$(i)$ $X(A)$ is dense in $X$, i.e. for any open $A$-subscheme $U\subseteq X$
one has $U(A)\neq\emptyset$.

$(ii)$ $X(A)=V(A)V(A)$.

$(iii)$ if $A$ is local, then for any ideal $I$ of $A$
the natural homomorphism $X(A)\to X(A/I)$ is surjective.
\end{lem}
\begin{proof}
The claim $(i)$ is clear since $(U\cap V)(A)\neq\emptyset$. The claim $(ii)$ is proved exactly as in the field
case. The claim $(iii)$ follows from $(ii)$.
\end{proof}


\begin{lem}\label{lem:Ueta}
Let $R$ be a local domain with an infinite residue field and the field of fractions $K$.
Let $G,H,T$ and $G',H',T'$ be reductive group schemes
over $R$, such that $H$ and $H'$ are $R$-rational, and there are two short exact
sequences of $R$-group scheme homomorphisms~\eqref{eq:seq}
and
\begin{equation}\label{eq:seq'}
1\to G'\to H'\xrightarrow{\mu'} T'\to 1.
\end{equation}
Assume also that
$\beta:T'\to T$ is a $R$-group scheme homomorphism such that
\begin{equation}\label{eq:mumu}
\beta(\mu'(H'(F)))\subseteq \mu(H(F))
\end{equation}
for any field extension $F$ of $K$.
Then there exists an open dense $R$-subscheme $U\subseteq H'$ and an $R$-morphism $\eta:U\to H$ such that
$\mu\circ\eta=\beta\circ\mu'|_U$, $1_{H'}\in U(R)$, and $\eta(1_{H'})=1_H$.
\end{lem}
\begin{proof}
The case where $R=K$ is a field was settled in~\cite[Lemma 3.1]{ChM-su} under the assumption that
$H'(R)$ is dense in $H'$.
The proof of the general case is obtained by reproducing the proof of that lemma verbatim,
taking into account that $H'(R)$ is dense in $H'$ by Lemma~\ref{lem:rat-surj} (note Remark~\ref{rem:ChM}).

\end{proof}

\begin{lem}\label{lem:K1transf}
\ \\
\indent $(i)$ In the setting of Lemma~\ref{lem:Ueta}, the natural map
$\bar\eta:U(R)\to H(R)/\bigl(RG(K)\cap H(R)\bigr)$ induced by $\eta$
extends uniquely to a homomorphism
$$
\tilde\beta:H'(R)\to H(R)/\bigl(RG(K)\cap H(R)\bigr)
$$
such that $\tilde\beta\bigl(RG'(K)\cap H'(R)\bigr)=1$. This homomorphism is independent of the
choice of a pair $(U,\eta)$ satisfying the claim of Lemma~\ref{lem:Ueta}. If $R=K$ is a field, then $\tilde\beta$
coincides with the map of the same name constructed in~\cite[Lemma 3.2]{ChM-su}, assuming Remark~\ref{rem:ChM}.

$(ii)$ Assume in addition that $G$ is simply connected, and $G$ and $G'$ possess
strictly proper parabolic $R$-subgroups that we denote by $P$ and $P'$ respectively.
If   $K_1^{G,P}(R)\to K_1^{G,P}(K)$ is injective,
then $\tilde\beta$ induces a homomorphism
$$
\hat\beta:K_1^{G',P'}(R)\to K_1^{G,P}(R).
$$
For any ring homomorphism $f:R\to S$, where
$S$ is a local domain with
the fraction field $E$ such that $K_1^{G,P}(S)\to K_1^{G,P}(E)$ is injective,
the map $\hat\beta$ is functorial with respect to $f$.
\end{lem}
\begin{proof}
$(i)$ By~\cite[Lemma 3.2]{ChM-su} the natural map $\bar\eta:U(K)\to H(K)/RG(K)$ induced by $\eta$
extends uniquely to a homomorphism
$$
\tilde\beta:H'(K)\to H(K)/RG(K)
$$
such that $\tilde\beta(RG'(K))=1$. In fact,
$\tilde\beta$ is given by the following formula~\cite[Proposition 1.4]{ChM-su}:
for any $g\in H'(K)$ and any $g_1,g_2\in U(K)$
such that $g=g_1g_2$ set
\begin{equation}\label{eq:g1g2}
\tilde\beta(g)=\bar\eta(g_1)\bar\eta(g_2).
\end{equation}
It is proved on~\cite[p. 183]{ChM-su} that $\tilde\beta$ is correctly defined and independent of the choice
of the pair $(U,\eta)$.
By Lemma~\ref{lem:rat-surj} for any $g\in U(R)$ we can find $g_1,g_2\in U(R)$ such that $g=g_1g_2$, therefore,
this homomorphism $\tilde\beta$ restricts to a homomorphism
$$
\tilde\beta:H'(R)\to H(R)/\bigl(RG(K)\cap H(R)\bigr)\le H(K)/RG(K),
$$
which is also correctly defined and independent of $(U,\eta)$.

$(ii)$ First we show that
\begin{equation}\label{eq:beta-GR}
\tilde\beta\bigl(G'(R)\bigr)\subseteq G(R)/\bigl(RG(K)\cap G(R)\bigr).
\end{equation}
Assume that for $g_1,g_2\in U(R)$ we have $g_1g_2\in G'(R)$. Then
$$
\mu\bigl(\eta(g_1)\eta(g_2)\bigr)=\mu(\eta(g_1))\mu(\eta(g_2))=\beta(\mu'(g_1))\beta(\mu'(g_2))=\beta(\mu'(g_1g_2))=1.
$$
Therefore, $\eta(g_1)\eta(g_2)\in G(R)$. This implies~\eqref{eq:beta-GR}.

Since $K_1^{G,P}(R)$ injects into $K_1^{G,P}(K)$, and
$E_{P}(K)=RG(K)$ for any strictly proper parabolic $R$-subgroup $P$ of $G$ by~\cite[Th\'eor\`eme 7.2]{Gil},
we have $RG(K)\cap G(R)=E_{P}(R)$.
Since $E_{P'}(R)\le E_{P'}(K)\le RG'(K)$, and $\tilde\beta(RG'(K))=1$, the map $\tilde\beta$
induces a correctly defined homomorphism
$$
\hat\beta:K_1^{G',P'}(R)=G'(R)/E_{P'}(R)\to G(R)/E_{P}(R)=K_1^{G,P}(R).
$$
The functoriality of $\hat\beta$ readily follows from~\eqref{eq:g1g2}.
\end{proof}

Denote by $\Groups$ the category of groups.

\begin{lem}\label{lem:Aring}
Let $k$ be a field, let $A$ be an essentially smooth $k$-domain. Let $F_1:\Sm_A\to \Groups$ be a presheaf, and let
$F_2:\Sm_A\to \Groups$ be a Nisnevich sheaf. Assume that
for any essentially smooth $A$-domain $R$ with the field of
fractions $E$ the homomorphism $F_2(R)\to F_2(E)$ is injective, and if $R$ is also Henselian local,
then there exists a homomorphism
$$
\lambda_R:F_1(R)\to F_2(R),
$$
functorial with respect to any homomorphism of essentially smooth local Henselian $A$-domains $R$.
Then there is a unique group homomorphism $\lambda_A:F_1(A)\to F_2(A)$ such that for any henselization
$A_p^h$ of a localization  of $A$ at a prime ideal $p$ the following diagram commutes:
\begin{equation}\label{eq:Aph}
\xymatrix@R=15pt@C=50pt{
F_1(A)\ar[d]^{F_1(i_p^h)}\ar@{-->}[r]^{\lambda_A}&F_2(A)\ar[d]^{F_2(i_p^h)}\\
F_1(A_p^h)\ar[r]^{\lambda_{A_p^h}}&F_2(A_p^h).\\
}
\end{equation}
Here $i_p^h$ is the natural inclusion $A\to A_p^h$.

\end{lem}
\begin{proof}
As usual, we extend the presheaves $F_1$ and $F_2$ by continuity to the category of Noetherian
essentially smooth $A$-schemes. This extension is well-defined thanks
to~\cite[Corollaire 8.13.2]{EGAIV-3}. The extension of $F_2$ is a Nisnevich sheaf
e.g. by~\cite[\href{http://stacks.math.columbia.edu/tag/00XI}{Tag 00XI}]{Stacks}.

Let $K$ be the fraction field of $A$. Set
$$
\theta=F_1(i_0^h):F_1(A)\to F_1(K).
$$
For any prime ideal $p$ of $A$, let $i_p:A_p\to K$ be the natural homomorphism.
By the assumption of the lemma, the natural map $F_2(i_0^h)=F_2(i_0):F_2(A)\to F_2(K)$ is injective, and for any
$p\in\Spec A$
the map $F_2(i_p):F_2(A_p)\to F_2(K)$ is injective; we identify $F_2(A)$
and $F_2(A_p)$ with the respective subgroups of $F_2(K)$. Clearly, this implies that once $\lambda_A$
exists, it is unique.

We construct $\lambda_A$ as follows. First we show that
\begin{equation}\label{eq:iF1F2}
\lambda_K\circ \theta\bigl(F_1(A)\bigr)\subseteq F_2(A)
\end{equation}
inside $F_2(K)$. Then we set
\begin{equation}\label{eq:lambda-def}
\lambda_A(x)=\lambda_K(\theta(x))\quad\mbox{for any}\quad x\in F_1(A),
\end{equation}
and we show that all diagrams~\eqref{eq:Aph} are commutative. Note that this definition of $\lambda_A$ automatically implies
that the diagram
\begin{equation}\label{eq:Ap}
\xymatrix@R=15pt@C=40pt{
F_1(A)\ar[d]^{F_1(i_p)}\ar[r]^{\lambda_A}&F_2(A)\ar[d]^{F_2(i_p)}\\
F_1(A_p)\ar[r]^{\lambda_{A_p}}&F_2(A_p).\\
}
\end{equation}
is commutative for any $p\in\Spec A$.

We prove~\eqref{eq:iF1F2} and the commutativity of the diagrams~\eqref{eq:Aph} by induction on dimension of $A$.
If $\dim A=0$, then $A=K$ and the statement is trivially true. Assume that $\dim A>0$.
Note that
$$
\lambda_K\circ \theta(F_1(A))\subseteq \lambda_K\circ F_1(i_p)\bigl(F_1(A_p)\bigr)
$$
for any $p\in\Spec A$,  and, clearly,
$$
F_2(A)=\bigcap_{p\in\Spec A}F_2(A_p)\subseteq F_2(K).
$$
 Therefore, in order to prove~\eqref{eq:iF1F2} and the commutativity of~\eqref{eq:Aph} (provided
the commutativity of~\eqref{eq:Ap}), we can assume right away that
$A=A_p$ is an essentially smooth
local $k$-domain with a maximal ideal $p$.

Fix $x\in F_1(A)$, and set $y=\lambda_K(\theta(x))\in F_2(K)$. Let
$\phi^h:A\to A^h$ be the henselization of $A$.
Set $y^h=\lambda_{A^h}(\phi^h(x))\in F_2(A^h)$. The functoriality of $\lambda_R$ for Henselian local rings $R$
together with the sheaf property
of $F_2$ implies that $y\in F_2(K)$ and $y^h\in F_2(A^h)$ are mapped to the same element in
$F_2(K\otimes_A A^h)$, since $\theta(x)$ and $\phi^h(x)$ are mapped to the same element in $F_1(K\otimes_A A^h)$.
Let $\phi':A\to A'$ be an \'etale local ring homomorphism, such that $\phi^h$ factors
through $A'$, $y^h$ lifts to $y'\in F_2(A')$, and $y$, $y'$ are still mapped
to the same element in $F_2(K\otimes_A A')$.

Let $f_1,\ldots,f_n\in p$ be such that
$$
\Spec(A)\setminus p=\bigcup_{i=1}^n\Spec (A_{f_i}).
$$
Then $\Spec (A_{f_i})$, $1\le i\le n$, and $\Spec(A')$ together form a Nisnevich cover of $\Spec(A)$.

Note that since $A$ is local, $\dim(A_{f_i})<\dim(A)$ for any $i$, and thus $A_{f_i}$'s satisfy
the induction hypothesis with respect to
the presheaves $F_1|_{\Sm_{A_{f_i}}}$ and $F_2|_{\Sm_{A_{f_i}}}$. Consequently,
$$
y\in F_2(A_{f_i})\subseteq F_2(K)\quad\mbox{for any}\quad 1\le i\le n.
$$
Let $L_i$ be the fraction field of $A_{f_i}\otimes_A A'$.
Since $A'$ is a flat $A$-module, the map $A_{f_i}\otimes_A A'\to K\otimes_A A'$ is injective,
and, clearly, the map $A_{f_i}\otimes_A A'\to L_i$ factors through it. Since by the assumption of the lemma
the map $F_2(A_{f_i}\otimes_A A')\to F_2(L_i)$ is injective, the map
$$
F_2(A_{f_i}\otimes_A A')\to F_2(K\otimes_A A')
$$
is also injective. It follows that $y\in \bigcap\limits_{i=1}^n F_2(A_{f_i})$
and $y'\in F_2(A')$ are mapped to the same element in $F_2(A_{f_i}\otimes_A A')$
for every $i$. Since $F_2$ is a Nisnevich sheaf, it follows that $y$ and $y'$ have a common lift to $F_2(A)$.
Since $F_2(A)\to F_2(K)$ is injective by the assumption of the lemma, we conclude that $y\in F_2(A)$,
which proves~\eqref{eq:iF1F2}.
The above argument also implies the commutativity of the diagram
\begin{equation*}
\xymatrix@R=15pt@C=40pt{
F_1(A)\ar[d]\ar[r]^{\lambda_A}&F_2(A)\ar[d]\\
F_1(A^h)\ar[r]^{\lambda_{A^h}}&F_2(A^h),\\
}
\end{equation*}
which proves the commutativity of~\eqref{eq:Aph}.
\end{proof}

\begin{proof}[Proof of Theorem~\ref{thm:transf-gen}]
For any essentially
smooth field $K$ over $k$, the existence of the sequence~\eqref{eq:seq} implies that $K_1^G(K)$
is abelian by~\cite[Lemma 1.2]{ChM-su}. Hence by Lemma~\ref{lem:inj} the functor
$K_{1,Nis}^G$ on $\Sm_k$ takes values in $\Ab$.

Consider two essentially smooth $k$-algebras $A$ and $B$ and a finite flat generically \'etale ring homomorphism $f:A\to B$.
We set $T'=R_{B/A}(T_B)$, $G'=R_{B/A}(G_B)$, and $H'=R_{B/A}(H_B)$. Then, clearly, $G',H',T'$ are reductive
$A$-group schemes, $H'$ is $A$-rational, $G'$ is a simply connected reductive group,
and these groups form the short exact sequence~\eqref{eq:seq'} over $A$.
By Lemma~\ref{lem:K1G'} $G'$ contains a strictly proper parabolic $A$-subgroup and
for any essentially smooth $A$-algebra $R$ one has
$$
K_{1,Nis}^{G'}(R)=K_{1,Nis}^G(R\otimes_A B).
$$

For any $A$-algebra $R$, one can define a natural "norm" homomorphism
$$
N_{B/A}:T'(R)=T(R\otimes_A B)\to T(A),
$$
extending the usual norm in the field case, see~\cite[\S 2]{Pa-pur}.
Such norm homomorphisms are functorial with respect to arbitrary "base change"{} ring homomorphisms
$g:A\to A'$, i.e. the following diagram commutes~\cite[p. 5]{Pa-pur}:
\begin{equation*}
\xymatrix@R=15pt@C=60pt{
T(B)\ar[r]^{N_{B/A}}\ar[d]^{\id\otimes g}&T(A)\ar[d]^{g}\\
T(B\otimes_A A')\ar[r]^{N_{B\otimes_A A'/A'}}& T(A')\\
}
\end{equation*}
This norm homomorphism on $\Sm_A$ defines a homomorphism of $A$-group schemes
$$
N_{B/A}:T'\to T_A.
$$

Assume for the time being that $A$ is a domain, and let $K$ be its fraction field. We check that the functors $F_1=K_{1,Nis}^{G'}$
and $F_2=K_{1,Nis}^G$ on $\Sm_A$ satisfy the assumptions of Lemma~\ref{lem:Aring}. First, by Lemma~\ref{lem:inj}
for any local essentially smooth $A$-domain $R$ with the field of fractions $E$ the map $F_2(R)\to F_2(E)$
is injective. Second, assume that $R$ is, moreover, Henselian local. Then
$F_1(R)=K_1^{G'}(R)$ and $F_2(R)=K_1^G(R)$.
We claim that the two short exact sequences of $A$-groups~\eqref{eq:seq} and~\eqref{eq:seq'} and
the morphism $\beta=N_{B/A}:T'\to T_A$ after base change to $R$
satisfy all conditions of Lemmas~\ref{lem:Ueta} and~\ref{lem:K1transf}; then Lemma~\ref{lem:K1transf}
provides the map $\lambda_R=\hat\beta$ required in Lemma~\ref{lem:Aring}. Indeed, the only
thing to check is the condition~\eqref{eq:mumu} that we proceed to establish.
Since $f$ is generically \'etale,
$K\otimes_A B$ is a finite product of finite separable field extensions of $K$.
Let $F$ be any field extension of $K$. Then $F\otimes_A B=F\otimes_K (K\otimes_A B)$, hence
$F\otimes_A B$ is also a finite product $\prod E_i$ of finite separable field extensions $E_i$ of $F$.
Since the norm homomorphism is compatible with base change, the norm homomorphism
$$
N_{B/A}: T'(F)=T(F\otimes_A B)\to T(F)
$$
is nothing but the product of the norm maps $N_{E_i/F}$. Hence
$\beta=N_{B/A}$ satisfies~\eqref{eq:mumu} for any field extension $F$ of $K$,
since by the assumption of our theorem Merkurjev's norm principle holds over $F$. It remains to note
that since $R$ is an essentially smooth $A$-domain, any field extension $F$ of $E$ is also a field extension of $K$.

We have proved that the functors $F_1=K_{1,Nis}^{G'}$ and $F_2=K_{1,Nis}^G$ on
$\Sm_A$ satisfy all assumptions of Lemma~\ref{lem:Aring}. Applying the claim of this lemma,
we set
$$
f_*=\lambda_A:K_{1,Nis}^{G'}(A)=K_{1,Nis}^G(B)\to K_{1,Nis}^G(A).
$$

Note that if $f:A\to B$ is a finite separable extension of fields, then
$\lambda_A=\lambda_K:K_1^{G'}(K)\to K_1^G(K)$ is the map
$\hat\beta$ constructed in Lemma~\ref{lem:K1transf}. That map is the natural restriction to $K_1^{G'}(K)=G'(K)/RG'(K)$
of the map $\tilde\beta$ identical to the map contructed in~\cite[Lemma 3.2]{ChM-su}.
Since $G'=R_{B/A}(G)$, one readily sees that
the map $\hat\beta$ is identical to the norm map $N_{B/A}:G(B)/RG(B)\to G(A)/RG(A)$ of~\cite[p. 187]{ChM-su},
as required.

Now we drop the assumption that $A$ is a domain, and let
$A=\prod\limits_{i=1}^n A_i$ be the decomposition of $A$ into a product of domains. Set $B_i=B\otimes_A A_i$.
Then we apply the above definition of $f_*$ to each base change $f_i:A_i\to B_i$ of $f$ via $A\to A_i$.
Clearly, we have
$$
K_{1,Nis}^G(B)\cong\prod_{i=1}^nK_{1,Nis}^G(B_i)\quad\mbox{and}\quad K_{1,Nis}^G(A)\cong\prod_{i=1}^nK_{1,Nis}^G(A_i).
$$
Then we set
$$
f_*=\prod (f_i)_*:K_{1,Nis}^G(B)\to K_{1,Nis}^G(A).
$$


We show that the transfer map $f_*$ defined above is compatible with any base change $A\to A'$, where
$A'$ is another essentially smooth $k$-algebra. Set $B'=A'\otimes_A B$ and let $f':A'\to B'$ be
the corresponding finite flat homomorphism. By the commutativity of the diagrams~\eqref{eq:Aph} in Lemma~\ref{lem:Aring},
it is enough to prove this claim assuming that both $A$ and $A'$ are henselian local rings essentially
smooth over $k$. Then the transfer maps corresponding to $A\to B$ and $A'\to B'$ are both represented
by the maps of the form $\hat\beta$ constructed in Lemma~\ref{lem:K1transf}. Namely,
$$
f_*=\hat N_{B/A}:K_1^{G}(B)=K_1^{R_{B/A}(G_B)}(A)\to K_1^G(A)
$$
and
$$
f'_*=\hat N_{B'/A'}:K_1^{G}(B')=K_1^{R_{B'/A'}(G_{B'})}(A')\to K_1^G(A').
$$
Since $K_1^G(B')=K_1^G(A'\otimes_A B)=K_1^{R_{B/A}(G_B)}(A')$ and the norm homomorphisms are compatible with
base change~\cite[p. 5, (1)]{Pa-pur},
we conclude that the required diagram  for $f'_*$ and $f_*$ commutes by the last claim of Lemma~\ref{lem:K1transf}.

It remains to check that $f_*$ satisfies the property (1) in the statement of Theorem~\ref{thm:transf-gen}.
First, assume that
$B=B_1\times B_2$
is a product of two essentially smooth $k$-algebras, and let $f_1:A\to B_1$ and $f_2:A\to B_2$
be the natural maps. We need to show that $f_*=({f_1}_*,{f_2}_*)$. As in the previous case, we are reduced to the
case where $A$ is a henselian local ring, and then the result follows from Lemma~\ref{lem:K1transf} and the multiplicativity
property of norm maps~\cite[p. 5, (2)]{Pa-pur}.
Finally, the second claim of the property (1) follows immediately from the
normalization property of norm maps~\cite[p. 5, (3)]{Pa-pur}.
\end{proof}

\begin{proof}[Proof of Theorem~\ref{thm:transf}]
Note that we have established in Lemma~\ref{lem:inj} that $K_{1,Nis}^G$ is $\A1$-invariant.
If $k$ is finite or $G$ is of trialitarian type
${}^{3(6)}D_4$, then $G$ is quasi-split and hence $K_{1,Nis}^G$ is trivial and there is nothing to prove.
Indeed, it is trivial on fields
by e.g.~\cite[Th\'eor\`eme 6.1]{Gil}, and hence on all smooth $k$-schemes by Lemma~\ref{lem:inj}.

Assume from now on that $G$ is non-trialitarian and $k$ is infinite.
By~\cite[p. 189--190]{ChM-su} (see also~\cite{ChM} for more
details), for any simply connected semisimple algebraic group $G$ over $k$ of classical type $A_l$ or $D_l$
(as well as for groups of type $B_l$ and $C_l$)
there exists a rational reductive $k$-group $H$ and a $k$-torus $T$ such that $G$ fits into an exact sequence
of algebraic $k$-group homomorphisms
$$
1\to G\to H\xrightarrow{\mu} T\to 1.
$$
Also, for any field extension $F/k$, the homomorphism $\mu$ satisfies Merkurjev's norm principle over $F$.
Indeed, if $G$ is of type $A_l$, then this follows from~\cite[Theorem 1.1]{MeBa}. If $G$ is of type $D_l$,
then this follows from~\cite[Theorems 3.9 and 4.3]{M-norm} by~\cite[Remark 4.1]{ChM-su}
(cf.~\cite[Theorem 4.6]{ChM-su}).

Thus, we can apply Theorem~\ref{thm:transf-gen}.
Clearly, the transfer maps $f_*$ defined in that theorem
satisfy the properties (1), (2), and (3) of Defininition~\ref{def:transfer}.
The property (4) of Defininition~\ref{def:transfer}
is trivially true, since $K_{1,Nis}^G$ is $\A1$-invariant. Thus, Theorem~\ref{thm:transf} is proved.

\end{proof}

\begin{proof}[Proof of Corollary~\ref{cor:bir}]
By~\cite[Theorem 6.14]{Ross} combined with Theorem~\ref{thm:transf}, there is a Gersten-type exact sequence
of abelian groups
$$
1\to K_{1,Nis}^G(X)\to \bigoplus_{x\in X^{(0)}} i_{x*}(K_{1,Nis}^G)(X)\to
\bigoplus_{x\in X^{(1)}}i_{x*}\bigl((K_{1,Nis}^G)_{-1}\bigr)(X)\to\ldots
$$
Here dy definition $(K_{1,Nis}^G)_{-1}(X)=\coker\bigl(K_{1,Nis}^G(\Aff^1_X)\to K_{1,Nis}^G(\Gm_{,X})\bigr)$.
By Lemma~\ref{lem:inj} we have $K_{1,Nis}^G(\Aff^1_X)\cong K_{1,Nis}^G(\Gm_{,X})$, that is,
$(K_{1,Nis}^G)_{-1}$ is trivial.

\end{proof}

\section{Rigidity}

Let $k$ be an infinite field.
We prove a rigidity theorem for torsion $\Aff^1$-invariant presheaves with oriented weak transfers on $\Sm_k$.
Any such presheaf $\CF$ extends by continuity to the category of (Noetherian) essentially smooth $k$-schemes, which are
filtered projective limits of smooth $k$-schemes. This extension is well-defined thanks
to~\cite[Corollaire 8.13.2]{EGAIV-3}.


The main result of this section is the following theorem.

\begin{thm}\label{thm:rigidity-main}
Let $k$ be an infinite field. Let $\CF$ be an $\Aff^1$-invariant presheaf of abelian groups on $\Sm_k$ with oriented
weak transfers for affine varieties. Assume also that
$l\CF=0$ for an integer $l\in k^\times$. Let $R=\mathcal{O}_{X,x}^h$ be the henselization of the local
ring of a smooth $k$-scheme $X$ at a point $x\in X(k)$, and let $K$ be the field of fractions of $R$.
If $\CF(R)\to\CF(K)$ is injective, then
$$
\CF(k)\cong\CF(R).
$$
\end{thm}

The proof of Theorem~\ref{thm:rigidity-main} imitates the proof of~\cite[Theorem 4.4]{SuVo}, relying on the following
result instead of~\cite[Theorem 4.3]{SuVo}.

\begin{thm}\label{thm:HY3.1}
Let $k$ be an infinite field. Let $\CF$ be an $\Aff^1$-invariant presheaf of abelian groups on $\Sm_k$ with oriented
weak transfers for affine varieties. Assume also that
$l\CF=0$ for an integer $l\in k^\times$.
Let $R$ be a Henselian local ring essentially smooth over $k$, and
let $K$ be its field of fractions. Let $P$
be the unique closed point of $\Spec R$. Let $f:M\to\Spec R$
be a smooth affine morphism of constant relative dimension $1$.
Let $s_0, s_1 : \Spec R \to M$ be  two sections of $f$ such that $s_0 (P) = s_1 (P)$.
Then the two compositions $\CF(M) \xra{s_i^*} \CF(R)\to \CF(K)$, $i=0,1$, are equal.
\end{thm}

The proof of Theorem~\ref{thm:HY3.1} is essentially the same as that of~\cite[Theorem 3.1]{HY} which
in its turn is based on~\cite[proof of the Rigidity Lemma, pp. 66--68]{Gab92} and on some
ideas of~\cite{SuVo,PaY,Ya-II}. We redo only those steps of the above proofs that are sketchy or appeal to the properties
of $\Aff^1$-stably representable cohomology theories.

Recall that the relative Picard group $\Pic(X,Z)$, where $X$ is a scheme and $Z$ is a closed subscheme of $X$,
is by definition the set of isomorphism classes of
pairs $(L,\psi)$ where $L$ is a line bundle on $X$ and $\psi$ is a trivialization of $L|_Z$. We also
denote by $\Div(X,Z)$ the group of Cartier divisors on $X$ supported outside of $Z$.
For any divisor $D$ on $X$ we denote by $\mathcal{O}_X(D)$ the corresponding element of the Picard group.

The following theorem essentially recaptures~\cite[Theorem 3.3]{HY}.

\begin{thm}\label{thm:pairing}
Let $k,K$ and $\CF$ be the same as in Theorem~\ref{thm:HY3.1}.
Let $\sigma:C^\circ\hookrightarrow\Aff^n_K$ be a smooth affine curve over $K$ with trivial tangent bundle, let $C$ be the normalization of
its projective closure, and let $C_\infty=(C\setminus C^\circ)_{red}$. Let $P_0,P_1\in C^\circ(K)$
be two closed points. There is a bilinear pairing
$$
\left<-,-\right>\colon\Pic(C,C_\infty)\times \CF(C^\circ)\to \CF(K)
$$
such that
$$
\l<\mathcal{O}_C(P_0-P_1),-\r>=(P_0)^*-(P_1)^*
$$
as maps from $\CF(C^\circ)$ to $\CF(K)$.
\end{thm}
\begin{proof}
Following~\cite{HY}, denote by $\Div_s(C,C_\infty)$ the subgroup of $\Div(C,C_\infty)$ consisting
of separable divisors (i.e. divisors $D=\sum a_ix_i$ such that the points $x_i$ correspond to finite separable field
extensions $L_i/K$) supported outside of $C_\infty$. Let $\mathcal{M}$ be the subgroup of $\Div(C,C_\infty)$ consisting of divisors
given by meromorphic functions taking the value $1$ on $C_\infty$.
Set
$$
\widetilde{\Pic}(C,C_\infty)=\Div_s(C,C_\infty)/(\Div_s(C,C_\infty)\cap\mathcal{M}).
$$
By~\cite[Proposition 3.12]{HY} the natural map $\widetilde{\Pic}(C,C_\infty)\to\Pic(C,C_\infty)$
is an isomorphism, hence it is enough to construct the pairing for the group $\widetilde{\Pic}(C,C_\infty)$.

We denote by
$$
f:C\to\Spec K
$$
the structure morphism of $C$. Define the pairing
$$
\l<-,-\r>\colon \Div_s(C,C_\infty)\times  \CF(C^\circ)\to  \CF(K)
$$
as follows: for any divisor $D=\sum a_ix_i$, where $x_i:\Spec L_i\to C^\circ$ are closed points, and any
$\alpha\in  \CF(C^\circ)$
set
$$
\l<D,\alpha\r>=\sum a_i\cdot (f\circ x_i)_*^{\sigma\circ x_i}\bigl(x_i^*(\alpha)\bigr).
$$
Since $f\circ P_0=f\circ P_1=\id_K$, by property $(2)$ of Definition~\ref{def:transfer}
one has
\begin{equation*}
\l<P_0-P_1,\alpha\r>=P_0^*(\alpha)-P_1^*(\alpha)\quad\mbox{for any}\ \alpha\in \CF(C^\circ).
\end{equation*}

It remains to show that for any $D\in \Div_s(C,C_\infty)\cap\mathcal{M}$
the map
$$
\l<D,-\r>\colon \CF(C^\circ)\to  \CF(K)
$$
is trivial.
Following~\cite{PaY}, we say that a divisor $D=\sum a_ix_i$ is unramified, if all $a_i$ are equal to $\pm 1$.
By~\cite[Lemma 3.11]{HY} any divisor $D\in \Div_s(C,C_\infty)\cap\mathcal{M}$ can be written as
a sum of unramified divisors in $\Div_s(C,C_\infty)\cap\mathcal{M}$. Thus, we need to show
that $\l<D,-\r>$ is trivial for an unramified divisor $D$.

Let
$h:C\to\Pro_K^1$ be a meromorphic function such that $\divi(h)=D$ and $h$ takes value $1$ on $C_\infty$.
Choose an open affine neighbourhood $\tilde C^\circ$ of $\Supp(D)$ contained
in the open subscheme of $C^\circ$ given by $h\neq 1$.
Consider the natural morphism
$$
\tilde h=h|_{\tilde C^\circ}:\tilde C^\circ\to\Pro^1_K\setminus\{1\}\cong\Aff^1_K,
$$
and let $\tau=(\sigma,\tilde h):\tilde C^\circ\hookrightarrow\Aff^n_K\times\Aff^1_K$ be the corresponding closed embedding.
By~\cite[Remark 3.10]{HY} the embedding $\tau$ has trivial normal bundle.

Since $D$ is unramified, $\tilde h$ is \'etale over $\{0\}$ and $\{\infty\}$, and
the divisors $D_0$ and
$D_\infty$ representing the $0$-locus and the $\infty$-locus of $h$ are of the form
$D_0=\sum_i y_i$, $D_\infty=\sum_j z_j$. To simplify the notation, we identify $D_0$ and $D_\infty$ with their supports
in $\tilde C^\circ$.
Consider the following diagram:
\begin{equation*}
\xymatrix@R=20pt@C=35pt{
\CF(D_0)\ar[d]^{(\tilde h|_{D_0})_*^\tau}&\CF(\tilde C^\circ)\ar[r]^{(\sqcup z_j)^*}
\ar[l]_{(\sqcup y_i)^*}\ar[d]^{\tilde h_*^\tau}&
     \CF(D_\infty)\ar[d]^{(\tilde h|_{D_\infty})_*^\tau}\\
\CF(K)&\CF(\Pro_K^1\setminus\{1\})\ar[r]^{i_\infty^*}\ar[l]_{i_0^*}&\CF(K)\\
}
\end{equation*}
By property $(3)$ of Definition~\ref{def:transfer} the left hand and
right hand squares of this diagram are commutative. Since $\CF$ satisfies $\Aff^1$-invariance,
$i_0^*$ and $i_\infty^*$ are isomorphisms and coincide.
Let $x:\Spec L\to \tilde C^\circ$ be any point with $L/K$ separable and $\tilde h(x)=0$. Then the transfer maps
$(f\circ x)_*^{\sigma\circ x}$ and $(f\circ x)_*^{\tau\circ x}$ coincide by
property $(4)$ of Definition~\ref{def:transfer}. By $\Aff^1$-invariance the same is true if $\tilde h(x)=\infty$. Together with
property $(1)$ of Definition~\ref{def:transfer}, this implies that for any $\alpha\in\CF(C^\circ)$ one has
$$
\left<D_0,\alpha\right>=(\tilde h|_{D_0})_*^\tau\circ (\sqcup y_i)^*(\alpha)=
(\tilde h|_{D_\infty})_*^\tau\circ (\sqcup z_j)^*(\alpha)=\left<D_\infty,\alpha\right>.
$$
Hence $\l<D,-\r>=0$, as required.
\end{proof}

\begin{proof}[Proof of Theorem~\ref{thm:HY3.1}]
Let $P_i=(s_i)_K:\Spec K\to M_K$ be the maps induced by $s_i$, $i=0,1$, via the base change $\Spec K\to\Spec R$.
Clearly, it is enough to show that $P_0^*=P_1^*$ as maps $\CF(M_K)\to\CF(K)$.
Since $s_0,s_1$ coincide on the closed fiber of $M$, the points $s_0(P)$ and $s_1(P)$ belong to the same
irreducible component of $M$. Without loss of generality, we assume that $M$ is irreducible.
Then $P_i$ belong to the same
irreducible component of $M_K$; denote this component by $X$. Since $X$ has a rational $K$-point, it
is geometrically connected over $K$, and hence geometrically irreducible.

Let $\sigma:M\to \Aff^n_R$ be a regular closed embedding of $M$ into an affine space, and let $\bar M$
be the projective closure of $M$ in $\Pro_R^n$, so that $M$ is open in $\bar M$.
The argument on pp. 68--69 of~\cite{Gab92}
shows that $\mathcal{O}_{\bar M}(s_0-s_1)\in\Pic(\bar M,(\bar M\setminus M)_{red})$ is $l$-divisible for
any integer $l\in k^\times$. We  choose $l$ to be the torsion integer of $\CF$. Consequently, $\mathcal{O}_{\bar X}(P_0-P_1)$
is $l$-divisible in $\Pic(\bar X,(\bar X\setminus X)_{red})$, where $\bar X$ is the closure of $X$ in $\bar M_K$.

Choose an open affine neighbourhood $C^\circ$ of $\{P_0,P_1\}$ in $X$ such that the tangent bundle
is trivial when restricted to $C^\circ$; this is possible since the restriction of the tangent bundle
to the semilocalization at $\{P_0,P_1\}$ is trivial.
Clearly, in order to prove the theorem, it is enough to show that $P_i^*:\CF(C^\circ)\to\CF(K)$ coincide.

Note that the closure of $C^\circ$ in $\bar M_K$ coincides with $\bar X$.
Let $C\to\bar X$ be the
normalization of $C^\circ$. Since $\mathcal{O}_{\bar X}(P_0-P_1)$ is $l$-divisible
in $\Pic(\bar X,(\bar X\setminus X)_{red})$, we conclude that $\mathcal{O}_{C}(P_0-P_1)$ is $l$-divisible in
$\Pic(C,(C\setminus C^\circ)_{red})=\Pic(C,C_\infty)$.
Apply Theorem~\ref{thm:pairing} to $C^\circ=U$ and the points $P_i$. Then for any $\alpha\in\CF(U)$ one has
$$
P_0^*(\alpha)-P_1^*(\alpha)=\l<\mathcal{O}_C(P_0-P_1),\alpha\r>.
$$
Since $\CF$ is $l$-torsion,
we conclude that $P_0^*(\alpha)-P_1^*(\alpha)=0$ for any $\alpha\in\CF(C^\circ)$, as required.
\end{proof}

\begin{proof}[Proof of Theorem~\ref{thm:rigidity-main}]
Take $R=\mathcal{O}_{X,x}^h$, where $X$ is $k$-smooth and $x\in X(k)$. We can replace $X$
with an open affine neighbourhood $U$ of $x$ such that there is an \'etale morphism $\pi:U\to\Aff^n_k$ for
some $n\ge 0$. Then, clearly, $\pi(x)\in\Aff^n_k(k)$ and $R\cong\mathcal{O}_{\Aff^n_k,\pi(x)}^h$.
Now the claim of the theorem is proved by induction on $n$;
the case $n=0$ is clear.
Take a generic linear projection $p:\Aff^n_k\to\Aff^{n-1}_k$. Set $x'=p(\pi(x))$,
$R'=\mathcal{O}_{\Aff^{n-1}_k,x'}$, and let $K'$ be the fraction field of $R'$. Clearly, one can find a section $s$ of $p$
satisfying $s(x')=\pi(x)$. Let $i:R'\to R$ and $j:R\to R'$ be the homomorphisms induced by $p$ and $s$ respectively,
so that $i\circ j=\id_{R'}$. Then $i^*:\CF(R')\to\CF(R)$ is injective, and hence the composition
$$
\CF(R')\to\CF(K')\to\CF(K)
$$
is injective. Therefore, the map $\CF(R')\to\CF(K')$ is  injective. Applying the induction hypothesis to $R'$,
we conclude that $\CF(k)\cong\CF(R')$. The rest of the proof is exactly the same as the proof
of~\cite[Theorem 4.4]{SuVo}, relying on Theorem~\ref{thm:HY3.1} combined with the injectivity of the map $\CF(R)\to\CF(K)$
instead of~\cite[Theorem 4.3]{SuVo}.
\end{proof}

\begin{lem}\label{lem:Dtorsion}
Let $k$ be a field, and let $G$ be an isotropic simply connected simple algebraic group over $k$
of the Tits index ${}^1 D^{(d)}_{l,r}$ or ${}^2 D^{(d)}_{l,r}$, $r\ge 1$, in the sense of~\cite{Tits66}.
Then $K_{1,Nis}^G$ is $d$-torsion.
\end{lem}
\begin{proof}
The definition of Tits indices implies that the first Tits algebra $\beta_G(\omega_1)$ of $G$ in $\Br(k)$ contains
an Azumaya algebra $E$ over $k$ of degree $d$; see~\cite[the proof of Theorem 3]{PS-tind}. Then, clearly, for any field extension $K/k$
the class $\beta_{G_K}(\omega_1)$ contains $E_K$. Therefore, there is a separable field extension $K'/K$
of degree $d$ such that $E_{K'}$ is split and $\beta_{G_{K'}}(\omega_1)$ is trivial. Then
$G_{K'}\cong\Spin(q)$, where $q$ is an isotropic quadratic form over $K'$. By~\cite[Th\'eor\`eme 6.1]{Gil} one has
$K_1^G(K')=1$. Denote $i:\Spec K'\to\Spec K$ the canonical morphism. By~\cite[Proposition 4.4]{ChM-su}
for any $x\in K_1^G(K)$ one has
$$
i^*\circ i_*(x)=x^{[K':K]}=x^{d}.
$$
Since $K_1^G(K')=1$, this implies that $K_1^G(K)$ is $d$-torsion. Then Lemma~\ref{lem:inj} implies
that $K_{1,Nis}^G$ is $d$-torsion.
\end{proof}

\begin{proof}[Proof of Corollary~\ref{cor:rigidity}]
If the field $k$ is finite or $G$ is of trialitarian type, then we conclude that $K_{1,Nis}^G$ is trivial
as in the proof of Theorem~\ref{thm:transf}. Hence $K_1^G(R)=K_1^G(R/m)$.
Assume that $k$ is infinite; then $\kappa=R/m$ is infinite. The group $G$ has the Tits index
${}^1 D^{(d)}_{l,r}$ or ${}^2 D^{(d)}_{l,r}$, and by Lemma~\ref{lem:Dtorsion} $K_{1,Nis}^G$ is $d$-torsion.
By the classification of Tits indices~\cite{Tits66,PS-tind} one has $d=2^m$ for some $m\ge 0$, hence $d$
is coprime to $\Char k$.

Assume first that
$R$ is of the form $\mathcal{O}_{X,x}^h$,
where $X$ is a smooth $k$-scheme and $x\in X(k)$. By Lemma~\ref{lem:inj} the map
$K_{1,Nis}^G(R)\to K_{1,Nis}^G(K)$ is injective, where $K$ is the field of fractions of $R$. Then by
Theorem~\ref{thm:transf} combined with Theorem~\ref{thm:rigidity-main} we conclude that $K_{1,Nis}^G(k)\cong K_{1,Nis}^G(R)$.

Now let $R$ be an arbitrary Henselian local $k$-ring with a coefficient subfield $k\cong R/m$.
The embedding $k\to R$ is geometrically regular, since $k$ is perfect.
Therefore by Popescu's theorem~\cite{Po90,Swan} $R$ is a filtered direct limit of local
essentially smooth $k$-algebras $A_i$. Clearly, $k$ is the residue field for each $A_i$ as well. By the universal
property of henselization, we can replace
every such algebra $A_i$ with its henselization $A_i^h$.
Since $K_1^G$ commutes with filtered direct limits of $k$-algebras,
we have
$$
K_1^G(R)=\varinjlim_i K_1^G(A_i^h)=\varinjlim_i K_{1,Nis}^G(A_i^h).
$$
Since $K_1^G(k)=K_{1,Nis}^G(k)\cong K_{1,Nis}^G(A_i^h)$ by the previous case, we are done.
\end{proof}

\end{document}